\documentclass[screen]{aomart}
\usepackage[a4paper, left=1.2in, right=1.2in, top=1in, bottom=1in]{geometry}
\usepackage[english]{babel}

\makeatletter
\fancypagestyle{firstpage}{%
  \fancyhf{}%
  \if@aom@manuscript@mode
    \lhead{\begin{picture}(0,0)%
        \put(-21,-25){\usebox{\@aom@linecount}}%
      \end{picture}}%
  \fi
  \fancyfoot[R]{\footnotesize\thepage} % Page number on right
  \fancyfootoffset[RO]{1.2in}          % Align with right margin
}

\makeatletter
\fancypagestyle{firstpage}{%
  \fancyhf{}%
  \if@aom@manuscript@mode
    \lhead{\begin{picture}(0,0)%
        \put(-21,-25){\usebox{\@aom@linecount}}%
      \end{picture}}%
  \fi
}

\fancypagestyle{mainstyle}{%
  \fancyhf{}%
  \fancyhead[CE]{\footnotesize{\hspace{-1in}B. Forghani and J. Frisch}} % Even pages: Authors (left)
  \fancyhead[CO]{\footnotesize{\hspace{1in} Non--stability of Liouville measures under convex combinations}}  % Odd pages: Title (right)
  \fancyhead[LE]{\hspace{-.7in}\footnotesize\thepage} % Even pages: Page number (left)
  \fancyhead[RO]{\footnotesize\thepage\hspace{-.7in}} % Odd pages: Page number (right)
}
\makeatother
\makeatother

\title{Non--stability of Liouville measures under convex combinations}
\author[B. Forghani]{Behrang Forghani}
\address{
College of Charleston}
\email{forghanib@cofc.edu} 
\author[J. Frisch]{Joshua Frisch} 
\address{University of California San Diego}
\email{joshfrisch@gmail.com}

\copyrightnote{}

\usepackage{xargs}                      % Use more than one optional parameter in a new commands
\usepackage[colorinlistoftodos,prependcaption,textsize=tiny]{todonotes}
\newcommandx{\change}[2][1=]{\todo[inline,linecolor=none,backgroundcolor=gray!20!,bordercolor=blue,textcolor=blue]{#2}}

\numberwithin{equation}{section}

\makeatletter
\renewcommand\subsection{%
  \@startsection{subsection}{3}{\z@}{2.25ex \@plus1ex \@minus.2ex}%
  {-.5em}{\normalfont\normalsize\bfseries}}
\makeatother

\newtheorem{theorem}{\bf{Theorem}}[section]
\newtheorem{thm}{\bf{Theorem}}

\newtheorem{corollary}[theorem]{\bf {Corollary}}
\newtheorem{lemma}[theorem]{\bf {Lemma}}

\theoremstyle{remark}
\newtheorem{remark}[theorem]{\bf {Remark}}
\theoremstyle{definition}
\newtheorem{definition}[theorem]{\bf {Definition}}

\newcommand{\ep}{\varepsilon}
\newcommand{\ta}{\tau}
\newcommand{\al}{\alpha}

\newcommand{\tm}{\tilde\mu}
\newcommand{\tti}{\tilde\theta}
\newcommand{\tn}{\tilde\nu}
\newcommand{\ZZ}{{\mathbb{Z}}}
\newcommand{\RR}{{\mathbb{R}}}
\newcommand{\NN}{{\mathbb{N}}}
\newcommand{\PP}{{\mathbb{P}}}

\newcommand{\U}{\mathbb{U}}
\newcommand{\D}{\mathcal{D}}

\renewcommand{\P}{{\mathcal P}}

\let\oldsummation\sum
\renewcommand{\sum}{\displaystyle\oldsummation}

\begin{document}
\begin{abstract}
\noindent For every non-hyper-FC-central countable amenable group and every $k\geq 2$, we provide a sequence of symmetric, fully supported probability measures such that their convex combination is non-Liouville (that is it admits a non-constant bounded harmonic function, equivalently, the Poisson boundary is non-trivial) if and only if at least $k$ of them appear in the convex combination. Particularly, our result implies that the set of Liouville measures is not closed under convex combination, which answers a question of Kaimanovich. We also provide a similar result under the additional assumption of finite entropy for those  non-hyper-FC-central countable  groups with the property that every symmetric, finitely supported probability measure is Liouville. These groups are the only known non-trivial examples of countable groups that admit Liouville measures with finite entropy.  Examples include the lamplighter group over $\ZZ$ and $\ZZ^2$, and the infinite symmetric group of finite permutations on $\ZZ$.
\end{abstract}
\maketitle
\setcounter{tocdepth}{1}
\tableofcontents
\section{Introduction}
Let $G$ be a countable group equipped with a probability measure $\mu$. Unless otherwise stated, we assume a probability measure is non-degenerate (that is, the support of $\mu$ generates $G$ as a semigroup). A bounded function $f\colon G \to \RR$ is called \emph{$\mu$--harmonic} if for every $g\in G$,
\[
f(g) = \sum_{h\in G} f(gh)\,\mu(h).
\]
It is easy to see that constant functions are bounded harmonic functions. A probability measure $\mu$ on $G$ is called \emph{Liouville} when every bounded $\mu$--harmonic function is constant; equivalently, the Poisson boundary associated with the random walk on $G$ by the law of $\mu$ is trivial (see \cite{Furstenberg-semi, K00} for connections between harmonic functions and the Poisson boundary). Similarly, a group is called \emph{Liouville} when it admits a non-degenerate Liouville probability measure. A group is called \emph{Choquet--Deny} when it is Liouville for every probability measure. 

Blackwell \cite{Blackwell1955} proved that the integer lattice $\ZZ^d$ is Choquet--Deny, and a similar result was proved for abelian groups by Choquet and Deny~\cite{Choquet-Deny1960}. Dynkin and Maljutov~\cite{Dynkin-Maljutov61} showed that countable nilpotent groups are also Choquet--Deny. On the other hand, Furstenberg \cite{Furstenberg71} showed that every non-degenerate probability measure on a non-amenable group is non-Liouville and conjectured that every amenable group is Liouville. Rosenblatt~\cite{Rosenblatt1981} and Kaimanovich--Vershik~\cite{K-Vershik83} proved Furstenberg's conjecture independently. Kaimanovich and Vershik~\cite{K-Vershik83} provided the first examples of amenable groups (lamplighter groups) that are not Choquet--Deny. In other words, there exist groups that admit both Liouville and non-Liouville probability measures. They also conjectured that every group with exponential growth is not Choquet--Deny. Erschler provided more examples of Liouville groups, such as finitely presented solvable groups~\cite{Erschler2004} and groups of intermediate growth~\cite{Erschler2004a}, which are not Choquet--Deny. 

It turns out that the structure of the conjugacy classes of a group plays a crucial role in determining whether a countable group is Choquet--Deny. A group is called an \emph{ICC group} when every non-identity element has infinitely many conjugates. Jaworski \cite{Jaworski2004} showed that when a group does not have any ICC quotients (such groups are known as hyper-FC-central), then it is Choquet--Deny. The second author, in collaboration with Hartman, Tamuz, and Vahidi Ferdowsi~\cite{Josh19}, proved that every countable group with an ICC quotient is non-Liouville. Thus, a group is Choquet--Deny if and only if it is hyper-FC-central, thereby resolving the conjecture of Kaimanovich and Vershik. 

There has been much effort devoted to determining whether a group is Liouville for one, all, or a class of probability measures; see \cite{Bartholdi-Nekrashevych-Kaimanovich2013, Amir-Omer-Virag2013, Amir-Virag2014,Amir-Angel-Bon-Virag2016,Erschler-Kaimanovich,BNZheng25} and references therein. When a probability measure is non-Liouville, much research has focused on describing the Poisson boundary in terms of the geometric or algebraic properties of the given group (this is known as the boundary identification problem). For more details, see \cite{K00, Erschler-survey, Tianyi-ICM,Chawla-Forghani-Frisch-Tiozzo22p, frisch-Silva2024} and references therein. 

For finitely generated groups, it is an open question whether all symmetric finitely supported probability measures on a given finitely generated group are either all Liouville or all non-Liouville simultaneously; this is called the \emph{Stability Problem}. An immediate consequence of an affirmative answer to the Stability Problem is that the Liouville (or non-Liouville) property of finitely supported symmetric probability measures on a given finitely generated group is closed under finite convex combinations. This question naturally fits into a broader context. Kaimanovich posed the question of whether the set of Liouville probability measures of a given group is closed under finite convex combinations at the Dynamics of (Semi-)Group Actions Conference at the University of Lodz in Poland; see also the discussion under Problem~E in \cite{ForghaniKaimanovich2025}, which studied the analogue of the question in the context of the equivalent Poisson boundary. The answer to Kaimanovich's question is clear for non-amenable groups or hyper-FC-central groups. However, this question was not known for any amenable (non-hyper-FC-central) group.

The goal of this paper is to provide a negative answer to Kaimanovich's question. Moreover, we show that the convexity property of Liouville measures does not hold for every non-hyper-FC-central countable amenable group. We even prove a stronger statement, which shows that secret sharing surprisingly appears in the context of convex combinations of Liouville measures. This demonstrates that convex spaces of measures can have an extremely flexible range of possible Liouville subsets.

\begin{thm}[= Theorem~\ref{thm:main1}]\label{thm:A}
For every non-hyper-FC-central countable amenable group  and every $k\geq 2$, there exists a sequence of symmetric, non-degenerate probability measures 
such that any convex combination of $k-1$ of these measures is Liouville, and any other 
(finite or infinite) convex combination of them is non-Liouville. 
\end{thm}

Our proof uses both \emph{F{\o}lner sets}   and \emph{switching sets}, see Section~\ref{sec:FS} for their definitions. The existence of F{\o}lner sets is equivalent to the amenability of the group. Kaimanovich--Vershik \cite{K-Vershik83} used them to construct Liouville probability measures. On the other hand, switching sets were introduced by Hartman--Frisch-Tamuz--Vahidi Ferdowsi \cite{Josh19} to construct  non-Liouville measures. We inductively build our measures by intertwining measures coming from F{\o}lner sets and switching sets.  The proof  that the convex combinations are non-Liouville follows from the general framework of Erschler--Kaimanovich \cite{Erschler-Kaimanovich}. On the other hand, the Liouville property of the extreme points follows the strategy of Kaimanovich--Vershik \cite{K-Vershik83}. The main novelty comes from the construction of the measures themselves. This construction relies on probabilistic notions, which we call  \emph{long gaps} (see Definition~\ref{df:gap}) and \emph{fit times} (see Definition~\ref{df:fittime}). 

Liouville probability measures obtained using F{\o}lner sets typically have infinite entropy.  Probability measures with infinite entropy on groups often exhibit strange behavior. For example, Kaimanovich \cite{Kaimanovich83-examples}
provided examples of measures on group $L$ such that the Liouville property holds for left but not right random walks. Kaimanovich \cite{Kaimanovich2024} found a measure on a product group $G \times H$ with non-Liouville measures such that the projections to $G$ and $H$ are both Liouville measures. Alpeev \cite{alpeev2021examples, alpeev2024secret} showed that these results are possible if and only if $L, G,$ and $H$ are amenable non-hyper-FC-central groups. Chawla--Frisch \cite{chawla-Frisch2025} found probability measures on free groups whose Poisson boundary is larger than the geometric boundary. All of these results are impossible when we restrict to finite entropy probability measures \cite{K-Vershik83, erschler2022poisson,Chawla-Forghani-Frisch-Tiozzo22p}.

Perhaps surprisingly, our results are \emph{not} a purely infinite entropy phenomenon. It is not possible to extend our results to finite entropy measures for all amenable non-hyper-FC-central groups.
Erschler \cite{Erschler2004a} proved that lamplighter groups over $\ZZ^d$ for $d\geq 3$ are non-Liouville for all non-degenerate finite entropy measures. However, we establish an analog of Theorem~\ref{thm:A} under the assumption of finite entropy when a group is \emph{finitely Liouville}, that is, every symmetric finitely supported probability measure on the group is Liouville. Groups with subexponential growth, the lamplighter groups over $\ZZ$ and $\ZZ^2$, and the infinite symmetric group of finite permutations of $\NN$ are examples of finitely Liouville  groups. It is unknown whether it is possible to have a finite entropy Liouville measure on a non-finitely Liouville group (conditional on this being indeed impossible, the result would be a complete classification of such groups).

\begin{thm}[=Theorem~\ref{thm:entropy}]\label{thm:B}
For every finitely Liouville  non-hyper-FC-central group and $k\geq 2$,  there exists a sequence of symmetric, non-degenerate probability measures with finite entropy
such that any convex combination of $k-1$ of these measures is Liouville, and any other 
(finite or infinite) convex combination of them is non-Liouville.  
\end{thm}
By Derriennic \cite{Der80} or Kaimanovich--Vershik \cite{K-Vershik83}, a probability measure with finite entropy is Liouville if and only if its asymptotic entropy is zero; see Section~\ref{sec:entropy} for details and definitions.  Our proof of Theorem~\ref{thm:B} as Theorem~\ref{thm:A} uses Kaimanovich-Erschler \cite{Erschler-Kaimanovich} to prove non-Liouville property for convex combinations. However, the proof of the Liouville property uses a totally different approach. We show the asymptotic entropy is zero to prove Theorem~\ref{thm:B}. Our proof utilizes  the $\D$--metric, see Definition~\ref{df:dmetric}, which makes the map of assigning a probability measure to the entropy of its $n$-fold convolution continuous (in contrast with the $\ell^1$-norm, see Remark~\ref{re:cont}). We leverage the continuity of $\D$-metric to approximate the entropy of fully supported measures with  finitely supported ones, where we apply the finitely Liouville hypothesis. 

In the context of finitely generated groups, it would be extremely interesting, but probably very difficult, to investigate Theorem~\ref{thm:A} in the class of finitely supported probability measures. It is unknown which amenable groups admit a non-degenerate finitely supported Liouville probability measure. Moreover, the only known examples beyond Choquet--Deny groups are finitely Liouville groups. Thus, all known examples are covered by groups in Theorem~\ref{thm:B}. 
\subsection*{Structure of the paper:} In Section~\ref{sec:record}, we only consider probability measures on $\NN$ to recall basic results regarding record times. We also define long gaps and fit times, which will be employed for showing a probability measure is Liouville. Section~\ref{sec:FS} is devoted to F{\o}lner and switching sets, and we remind the general construction of non-Liouville measures for countable ICC  groups. In Section~\ref{sec:liouville}, we construct Liouville probability measures using long gaps and prove Theorem~\ref{thm:A}. In the last section, we recall basic properties of entropy, define $\D$--metric and  apply it to prove Theorem~\ref{thm:B}.

\subsection*{Acknowledgment}
B. Forghani was supported by NSF grant  DMS-2246727 and J. Frisch was supported by NSF grant DMS-2348981. 

 \section{Record Times, Fit Times, Long Gaps}\label{sec:record}
We are mainly interested in probability measures on $\NN$ that can be written as a convex combination of two measures, which at least one of them is fully supported on $\NN$. In this section, we introduce notions of record times, long gaps and fit times for probability measures on $\NN$. We will use these notions later for random walks on groups. 

\subsection*{Notation} We denote the values of an integer-valued function  or a probability measure $f$ on $\NN$ by $f_n$  or $f(n)$ interchangeably.

\subsection{Record Times} Let $\al$ be a probability measure on $\NN$. Let $(\xi_i)_{i\geq 1}$ be a sequence of independent and identically distributed
(i.i.d.) random variables according to $\al$. A record time is the first time in the sequence when the value exceeds all the values observed up to that time.  We define the sequence of record times $(\ta_n)_{n\geq1}$ and record values $(\xi_{\ta_n})_{n\geq1}$ inductively by $\tau_1=1$ and for $n\geq 2$
\begin{equation}\label{eq:record}
M_n= \max\{\xi_1,\dots, \xi_n\}, \ \ \ \tau_k = \min\{ i> \tau_{k-1}\ : \xi_i= M_i  \}.
\end{equation}
We will need the following lemma, which is an application of the classic Borel-Cantelli lemma, about controlling record times and record values. 

\begin{lemma}(\cite[Lemma~2.12, Lemma~217]{Erschler-Kaimanovich})\label{lem:phi}
Let  $\al$ be an infinitely supported probability measure on $\NN$. 
Then, there exist non-constant and non-decreasing functions $\varphi,\Phi:\NN \to \NN$ such that  almost surely
$$
\varphi_{n}  < M_n \,, \hspace{1cm} \tau_{n+1} \leq \Phi({\xi_{\ta_n}}) \,,
$$
for all sufficiently large  $n$. 
\end{lemma}
 \begin{definition}[\bf Ladder]
A non-decreasing function $\Phi$ in Lemma~\ref{lem:phi} is called  an $\alpha$--ladder. 
\end{definition}
We are interested in finding a universal ladder to control the record times when a probability measure can be written as a convex combination of two probability measures.
\begin{lemma}\label{lem:universal}
Let $\bar\al$  be a fully supported probability measure  on $\NN$. 
Then there exists a universal $\bar\al$-ladder $\Phi$, which is also an $\al$--ladder 
for any probability measure $\al$ which satisfies $d \bar\al\leq \al$ for some $0<d\leq 1$. 
\end{lemma}
\begin{proof}
Let $\bar{\xi}_1,\dots,\bar{\xi}_n$ be i.i.d. according to $\bar\al$, and $\xi_1,\dots,\xi_n$ be i.i.d. according to $\al$. By Lemma~\ref{lem:phi}, there exists $\bar\varphi$ such that almost every $\max\{\bar{\xi}_1,\dots,\bar{\xi}_n\} > \bar\varphi_n$
for sufficiently large $n$. Since $\al \geq d\bar\al$, we can rewrite $\al$ as a convex combination of $\bar\al$ and another probability measure; $\al = d\bar\al + (1-d)\al'$. Let $Z_n$ be the binomial random variable that
counts the number of samples
drawn according to $\bar\al$ in $n$ random independent samples according to $\al$. By the Law of Large Numbers, almost every $2 Z_n \geq dn$ for sufficiently large $n$. Therefore, 
$$
\max\{\xi_1,\dots,\xi_n\} \geq \bar{\varphi}(Z_n) \geq \bar{\varphi}(\lfloor{\frac{dn}{2}} \rfloor) \geq \bar{\varphi}(\lfloor \sqrt n \rfloor) 
$$
for sufficiently large enough $n$.
Thus it is enough to define $\varphi(n) = \bar\varphi(\lfloor \sqrt n \rfloor)$. We now can define $\Phi$ by the following condition:
$$
\{ (i,j) \in \NN \times \NN : \varphi_i<j \} = \{ (i,j) \in \NN \times \NN :i < \Phi_j \}\,.
$$
\end{proof}

\subsection{Long gaps and Fit times} We introduce two new notions of long gaps and fit times, which later will be employed to construct Liouville probability measures. 
\begin{definition}\label{df:alpha}
 Let   $p$ and $q$ be two infinitely supported probability measures on $\NN$  such that the support of $q$ is $\NN$. For $ 0< t < 1$, we define
 $$
 \al_{t,p,q} = tp + (1-t) q .
 $$
\end{definition}
 Let $(\xi_i)_{i\geq1}$ be a sequence of independent and identically distributed 
random variables according to $\al_{t,p,q}$.  We can study $\xi_i$ using a Bernoulli random variable $\zeta_i$ on $\{0,1\}$  such that 
$$
P(\xi_i=n, \zeta_i=1)=tp_n, \hspace{.5cm} P(\xi_i=n, \zeta_i=0)=(1-t)q_n, \hspace{.5cm} P(\zeta_i=1)=t.
$$
\begin{definition}[\bf Long gap]\label{df:gap}The probability measure $\al_{t,p,q}$  (Definition~\ref{df:alpha}) has  a \emph{long gap} for $q$ whenever  there are infinitely many $m$ such that 
 $$
t\sum_{i=m}^\infty p_i \leq (1-t) q_m.
 $$
 \end{definition}
 We will show below that ``conditional'' long gaps can exist. For a given $A \subseteq \NN$ with $p(A)>0$,  we define the probability measure $p^A$ as the conditional measure of $p$, given $A$. Thus, $p^A(n) = p_n/p(A)$ for every $n \in A$, and zero otherwise. 
 \begin{definition}[\bf $k$--Cover]\label{df:cover}
Let $k \in \NN$ and $(A_j)_{j\in J}$ be a family of subsets in $\NN$. We say $(A_j)_{j\in J}$ is a $k$--cover whenever for every distinct $i_1, \dots, i_k$ in $J$:
\begin{enumerate}
    \item $A_{i_1}\cup \dots \cup A_{i_k}=\NN$, 
    \item $\NN\backslash \Big(A_{i_1}\cup \dots \cup A_{i_{k-1}}\Big)$ is infinite.
\end{enumerate}
\end{definition}
We are interested in those $k$--covers that lead to constructing probability measures with long gaps. First, note that for every partition of $\NN$, one can easily construct a $k$--cover.
\begin{lemma}\label{lem:kcovering}
    Let $(D_i)_{i\geq 1}$ be a partition of $\NN$. Let $k \in \NN$ and $\psi: \NN \to \NN^k$ be a bijection. Define
    $$
    B_i=\Big\{(l_1,\dots, l_k) \in \NN^k \ : \   \  l_j =i \mbox{ for some } j\geq 2\Big\} \,.
    $$
Then $(A_i)_{i\geq 1} $ is a $k$--cover for 
$$
A_i = \bigcup_{\psi_j\not \in B_i} D_j \,.
$$
\end{lemma}
\begin{proof}
  Let $i_1, \dots, i_k$ be distinct. It is simple to check that $B_{i_1} \cap \cdots \cap B_{i_k} = \emptyset$.  For every $x \in \NN$ there exists $j$ such that $x \in D_j$, thus  at least one of  $B_{i_1}, \dots,  B_{i_k}$ doesn't contain $\psi_j$, which implies $A_{i_1}\cup \dots \cup  A_{i_k} = \NN$. It also is straightforward to check that  $B= B_{i_1} \cap \dots \cap B_{i_{k-1}}$ is an infinite set and $\NN^k \not = B$. Thus the second property of $k$--cover holds for $(A_i)_{i\geq1}$.  Hence,  $(A_i)_{i\geq 1} $ is a $k$--cover.
\end{proof}
\begin{lemma}\label{lem:conditionalgap}
Let $p$ and $q$ be two infinitely supported probability measures on $\NN$. 
For every $k\geq 2$, there exists $k$--cover $(A_j)_{j\geq 1}$  such that  every $k-1$ convex combination of probability measures $\al_j=t_j p^{A_j}+(1-t_j) q$ has a long gap for $q$ (Definition \ref{df:gap}) for any $0<t_j<1$. 
\end{lemma}
\begin{proof}
Since $q$ is infinitely supported and $\sum_{m= i}^\infty p_m \to 0$ as $i\to \infty$, there exists a strictly increasing sequence $(n_i)_{i\geq 1}$ such that $n_1=1$ and  for every $i\geq 1$
$$
\frac{q_{n_i} }{n_i}  \geq   \sum_{m = n_{(i+1)}}^\infty p_m  \ .
$$
We  define $D_i = [n_i, n_{(i+1)}) \cap \NN$. 
We construct $k$--cover $(A_j)_{j\geq1}$ by applying Lemma~\ref{lem:kcovering} to the partition $(D_i)_{i\geq 1}$.  Let $\al$ be a $k-1$ convex combination of $\al_j$.  Without loss of generality, suppose that $\al = b_1 \al_1 + \dots + b_{k-1} \al_{k-1}$ where $b_1+\dots+b_{k-1} = 1$ and $b_i\geq0$.
We will show that $\al$ has a long gap  for $q$. Since $A_j \not = \NN$, we have  $0<p(A_j)<1$ for every $j\in \NN$.  For every $ 0<t_m<1$ and $m=1,\dots, k-1$, there exists $n_{m_0}$ (depending on $t_1,\dots, t_{k-1}$)  such that 
$$
\frac{(1-t_m)p(A_m)}{t_m}>\frac{1}{n_{m_0}}  \,.
$$ 
Let  $i>m_0 $ and $n_i \not\in (A_1 \cap \dots \cap A_{k-1})$ (which is an infinite set because $(A_j)_{j\geq 1}$ is a $k$--cover), thus $p^{A_m}(D_i)=0$ for infinitely many $i$ and for every $m=1,\dots,k-1$:
\begin{equation}\label{eq:kcover}
 (1-t_m) q_{n_i} \geq   \frac{q_{n_i}}{n_i}\frac{t_m}{p(A_m)}\geq  \frac{t_m}{p(A_m)} \sum_{k = n_{(i+1)}}^\infty p_k = t_m\sum_{k = n_i}^\infty p^{A_m}_k\,.
\end{equation}
Multiplying by $b_m$ both sides of inequality~(\ref{eq:kcover}) and adding them from $m=1$ to $k-1$ imply $\al$ has a long gap for $q$. 
\end{proof}

\begin{definition}[\bf Fit time]\label{df:fittime}
Let $(\xi_i,\zeta_i)_{i\geq 1}$  be a sequence of i.i.d. random variables according to $\al_{t,p,q}$ (Definition~\ref{df:alpha}). Suppose that $\Phi$ is an $\alpha_{t,p,q}$--ladder as in Lemma~\ref{lem:phi}. We say   $(\xi_i,\zeta_i)_{i\geq 1}$ has a \emph{fit time} at time $n$   for $m$ whenever the following are satisfied:
\begin{enumerate}
\item $n<\Phi_m$, 
\item $(\xi_n, \zeta_n)=(m, 0)$,
\item  $\xi_i<m$ for all $i<n$.
\end{enumerate}
\end{definition}
\begin{theorem}\label{them:longfit}
Let $(\xi_i,\zeta_i)_{i\geq 1}$ be a sequence of i.i.d. random variables  according to $\al_{t,p,q}$ (Definition~\ref{df:alpha}). If $\al_{t,p,q}$  has a long gap (Definition~\ref{df:gap}) for $q$, then $(\xi_i,\zeta_i)_{i\geq 1}$ has infinitely many fit times. 
\end{theorem}
\begin{proof}
Since $\al_{t,p,q}$ is infinitely supported, for almost every $(\xi_i,\zeta_i)_{i\geq 1}$  there is a strictly increasing sequence 
of recording times $(\tau_n)_{n\geq 1}$ that almost surely satisfies    $\tau_n<\Phi(\xi_{\tau_n})$ for sufficiently  large enough $n$ by Lemma~\ref{lem:phi}.  Also, the record values eventually converge to infinity as $\al_{t,p,q}$ is infinitely supported. Therefore, it is sufficient to show that 
when $\al_{t,p,q}$ has a long gap for $q$, then for almost every $(\xi_i,\zeta_i)_{i\geq 1}$  there are infinitely many record times $\tau_n$  such that $\zeta_{\tau_n}=0$. Therefore, it is enough to show that $E_m$, the event consisting of all $(\xi_i,\zeta_i)_{i\geq1}$ such that $(m,0)$ appears before $\{(m,1), (m+1,1),\dots\}$ up to a finite time, happens  for infinitely many $m$ with probability one.  We have
\begin{equation}\label{eq:Em}
E_m = \bigcup_{n=1}^\infty E_{n,m},
 \end{equation} 
where $E_{n,m}$  is  the set of all  $ (\xi_i,\zeta_i)_{i\geq 1}$ that  their  $n$-th increment is $(m,0)$ and the first $n-1$ increments are not in  $\{(j,1) :  \ j \geq m \}$: 
$$
E_{n,m} = \Big\{ (\xi_i,\zeta_i)_{i\geq 1}\ :\ (\xi_n, \zeta_n)=(m,0),   \forall i<n\  \xi_i<m \Big\}.
$$ 

Because $\al_{t,p,q}$ has a long gap for $q$, there are infinitely many $m$ such that for every $i\geq 1$
\begin{equation}\label{eq:manygap}
P(\xi_i=m, \zeta_i=0) \geq \sum_{n=m}^\infty P(\xi_i=n,\zeta_i=1).
\end{equation}
 Hence, $\al_{t,p,q}$ having a long gap for $q$ implies that $m$ satisfies (\ref{eq:manygap}) the probability of an increment of $(m,0)$ shows up before an increment from  $\{(j,1) :  \ j \geq m \}$ is higher than showing up after increments from  $\{(j,1) :  \ j \geq m \}$. Thus,
$P(E_m) \geq 1- P(E_m)$ and $P(E_m) \geq 1/2$ for infinitely many $m$. Therefore, 
$$
P(\limsup_{m\to \infty} E_m) \geq \frac12 \,.
$$ 
 Note that $\displaystyle\limsup_{m\to \infty} E_m$ is a tail event with positive probability, thus the desired result follows from a zero-one lemma .  
\end{proof}

\section{Switching and F{\o}lner Sets}\label{sec:FS}
In this section, we first remind readers about the definitions of   switching and F{\o}lner sets.     
\begin{definition}[\bf   Switching  Sets]
Let $B$ be a non-empty subset of $G$ and $g$ in $G$. We say $g$ is a \emph{  switching} element for $B$ when 
for $i,j \in \{1,-1\}$
$$
g^i B g^j \cap B \subseteq \{e\}.
$$
When $S$ is a non-empty subset of $G$, we write that $S$ is \emph{switching} for $B$ when $g$ is   switching  for every $g\in S$. 
\end{definition}
\begin{remark}\label{re:switchin} Note that our definition of the switching element is equivalent to the \emph{super switching} elements in \cite{Josh19}. We also note that by \cite[Proposition 2.5]{Josh19}  every non-empty symmetric finite subset of a countable amenable ICC group has infinitely many   switching elements. The same result is proved by
\cite[Proposition 4.25]{Erschler-Kaimanovich}  for all countable ICC groups.
\end{remark}

\begin{definition}[\bf F{\o}lner property] Let $K$ and $F$ be finite subsets of $G$. For $\ep>0$, we say  $F$ is $\ep$--F{\o}lner for $K$ if
$$
|gF \Delta F| < \ep |F|   
$$ 
for all $g \in K$.
\end{definition}
\begin{remark}\label{re:folner}
Note that $G$ is amenable if and only if for every $\ep>0$ and  finite subset $K$ of $G$, there exists an $\ep$--F{\o}lner set for $K$. Moreover,  by  \cite[Theorem 5.2]{Namioka}, one can choose $F$ to be symmetric and include any given finite subset $A$ of $G$, see also \cite[Theorem 1"]{Emerson}. 
\end{remark}
\begin{definition}[\bf F{\o}lner--switching sets] 
Let $(F_n)_{n\geq 1}$ and $(S_n)_{\geq 1}$ be sequences of subsets in $G$ such that
$S_n=\{s_n,s_n^{-1}\}$ and $e\in F_1$. Suppose that   $\Phi:\NN \to \NN$ is non-decreasing. We say $(F_n,S_n,\Phi_n)_{n\geq 1}$ is \emph{F{\o}lner--switching} when  
\begin{enumerate}
\item $F_n$ is symmetric,
\item $(F_n)_{n\geq 1}$ is pairwise disjoint,
\item  $F_n$ is  $\frac{1}{n}$--F{\o}lner for $\Big(S_1\cup \cdots \cup S_{n-1} \cup F_1 \cup \cdots \cup F_{n-1} \Big)^{5\Phi_{n} }$,
\item $S_{n}$ is   switching for $\Big(S_1 \cup \cdots \cup S_{n-1} \cup F_1\cup \cdots\cup F_{n}\Big)^{5\Phi_{n}}$,
\item $G=\displaystyle\bigcup_{n=1}^\infty F_n$.
\end{enumerate}
\end{definition}
Note that since $e\in F_1$ and $\Phi$ is non-decreasing, property (4) and definition of  switching elements imply that $(S_n)_{n\geq 1}$ is pairwise disjoint, and also  $F_n$ and $S_n$ are pairwise disjoint.
\begin{lemma}\label{lem:folner-switch}
Let $G$ be a countable amenable ICC group. If $\Phi:\NN \to \NN$ is non-decreasing, then there exists F{\o}lner--switching  $(F_n,S_n,\Phi_n)_{n\geq 1}$.  
\end{lemma}
\begin{proof}
The proof is an immediate consequence of amenability and existence of   switching elements in ICC groups.  Let us denote by $\ep_n=\frac1n$ and 
$G=\{a_i, a_i^{-1} :\ i\geq 1\}$ such that $a_1=e$ is the identity element of $G$. Suppose that $s_1\not = e$. We define $F_1=\{e\}$ and  $S_1=\{s_1, s_1^{-1}\}$. It is clear that $F_1$ and $S_1$ satisfy   properties (1)-(5). 
By induction, suppose that $(F_i,S_i)_{1 \leq i<n}$ are constructed that satisfy the desired properties (1)--(4). Let $K_i=S_1\cup\cdots\cup S_{i} \cup F_1\cdots \cup F_{i}$. Since $G$ is amenable, by Remark~\ref{re:folner} there exists a symmetric finite set $F_n$ which is $\ep_n$--F{\o}lner for $(K_{n-1})^{5\Phi_{n}}$, moreover, $F_n$ can be chosen that $F_n \cap (F_1\cup\dots \cup F_{n-1})=\emptyset$  and $a_{i_n} \in F_n$, where $i_n = \min\{ i\ :\ a_i \not \in F_1\cup\dots \cup F_{n-1} \}$.
Since $G$ is an ICC group, by Remark~\ref{re:switchin} $(K_{n-1} \cup F_n)^{5\Phi_n}$ has infinitely many switching elements. Thus, 
there exists $S_n=\{s_{n},s^{-1}_{n}\}$ which is switching for $(K_{n-1} \cup F_n)^{5\Phi_n}$. Thus, the desired result is obtained by induction.
\end{proof}
\subsection{Construction of non-Liouville measures}\label{sec:construction}
We begin by recalling the following result, which will be used to construct non-Liouville measures on an ICC group.
\begin{theorem}{\cite[Theorem 3.10]{Erschler-Kaimanovich}}\label{thm:erchler-kaimanovich}
Let $G$ be a countable ICC group and $\mu$ a non-degenerate probability measure on $G$.  Assume that $\alpha$ is an infinitely supported probability measure on $\NN$ and $\Phi$ is an $\al$--ladder.
Suppose that $(G_n)_{n\geq 1}$ and $(S_n)_{n\geq 1}$ are  sequences of finite sets in $G$ such that $G=\cup_{n=1}^\infty G_n$ and $e\in G_1$ and $S_n$ is switching for $(S_1 \cup \dots \cup S_{n-1} \cup G_1 \cup \dots \cup G_n)^{5\Phi_n}$.  Let $\mu(G_n \cup S_n) = \al_n$ for every $n\geq 1$ and
\begin{equation}\label{eq:simple1}
\sum_{n=1}^\infty\Big(\frac{\al_n}{\al_n+\al_{n+1}+\cdots}\Big)^2 <\infty \,, \hspace{1cm}
\sum_{n=1}^\infty \frac{\mu(G_n)}{\al_n} < \infty \,.
\end{equation}
Then, the probability measure $\mu$ is non-Liouville. 
\end{theorem}
\begin{remark}
 The first condition in (\ref{eq:simple1}) ensures that the probability measure 
$\al$ on $\NN$
eventually admits simple records—that is, for sufficiently large record values, each occurs exactly once. The second condition ensures that these simple records eventually arise from switching sets. This is the key point at which the measure becomes non-Liouville in this construction, see \cite{Erschler-Kaimanovich} for more details.
\end{remark}

We use the construction in Theorem~\ref{thm:erchler-kaimanovich} for two different cases. In the first case, we are interested in using F{\o}lner--switching sets, see Section~\ref{sec:liouville}.   In the second case, we are interested in finite entropy measures, see Section~\ref{sec:entropy}

We denote the cardinality of $F$ by $|F|$. Suppose that $F$ is a finite set, denote by $\U(F)$ the uniform probability measure on $F$:
$$
\U(F)(g)= \begin{cases}
 \frac{1}{|F|} &  g\in F\,,\\
 0 & \mbox{ otherwise}\,.
\end{cases} 
$$

\begin{definition}\label{df:mu2}
    Let  $p$ and $q$ be two infinitely supported probability measures on $\NN$ such that $q$ is fully supported on $\NN$. Suppose that $\Phi$ is a $q$--universal ladder as in Lemma~\ref{lem:universal}. Let $G=\{a_i,  a_i^{-1}: i \in \NN\}$  such that $a_1=e$. Define $G_n=\{a_n,a_n^{-1}\}$. Suppose that $S_n=\{s_n,s_n^{-1}\}$ is switching for 
$(S_1\cup \cdots \cup S_{n-1}\cup G_1\cup\cdots \cup G_n)^{5\Phi_n}$. Let  $\U(G_n)$ and $\U(S_n)$ be uniform probability measures on $G_n$ and $S_n$, respectively. For $0<t<1$, we define the probability measure 
\begin{equation}\label{eq:lambdaconstruction}
\mu^{t,p,q}= t \sum_{n=1}^\infty p_n \U(S_n) + (1-t) \sum_{n=1}^\infty q_n\U(G_n) \,.
\end{equation}
\end{definition}
Note  $\mu^{t,p,q}$ is symmetric and fully supported. When $p$ and $q$ both have finite entropy, then $\mu^{t,p,q}$ has finite entropy. 

\begin{definition}\label{df:mu}
Let $q$ be a fully supported probability measure on $\NN$ and $\Phi$ be a universal $q$--ladder as in Lemma~\ref{lem:universal}. Let $(F_n,S_n,\Phi_n)_{n\geq 1}$ be a  F{\o}lner--switching, which exists by Lemma~\ref{lem:folner-switch}. Let $p$ be an infinitely supported probability measure on $\NN$. 
For $0<t <1$, we define the  probability measure 
\begin{equation}
\mu_{t,p,q} =t \sum_{n=1}^\infty p_n \U(S_n) +  (1-t) \sum_{n=1}^\infty q_n \U(F_n).
\end{equation}
\end{definition}
It is clear that $\mu_{t,p,q}$ is symmetric. Since $G=\cup_{n=1}^\infty F_n$ and $q$ is fully supported on $\NN$, the probability measure $\mu_{t,p,q}$ is fully supported. Note that by construction of F{\o}lner--switching sets, each $S_n$ has at most two elements. But, we cannot control the number of elements in $F_n$. Thus, the entropy of $\mu_{t,p,q}$ cannot be controlled and might be infinite.

Note that in both constructions above, the ladder depends on a measure. However, one can replace the measure with a class of measures that share the same ladder. 
\begin{lemma}\label{lem:nont}
Suppose that $p$ and $q$ are infinitely supported probability measures on $\NN$ such that
\begin{equation}\label{eq:pq}
 \sum_{n=1}^\infty\Big(\frac{p_n+q_n}{p_n+q_n+p_{n+1}+q_{n+1}+\cdots}\Big)^2 <\infty\,,  \hspace{1cm} 
\sum_{n=1}^\infty \frac{q_n}{q_n+p_n} < \infty \,. 
\end{equation}
If $\tilde{p}$ and $\tilde{q}$ are probability measures on $\NN$ such that $d(p+q) \leq \tilde{p} +\tilde{q}  \leq d' (p+q)$ for some $ 0<d \leq 1$ and $d'>0$, then  $\mu^{t,\tilde{p},\tilde{q}}$ (Definition~\ref{df:mu2}) and $\mu_{t,\tilde{p},\tilde{q}}$ (Definition~\ref{df:mu}) both are  non-Liouville for every $0<t < 1$.
\end{lemma}
\begin{proof}
By definitions of $\mu^{t,p,q}$ and $\mu_{t,p,q}$, the ladder $\Phi$ is a $q$--universal ladder. Since $(1-t)dq \leq (1-t)\tilde{q}+t\tilde{p}$, by Lemma~\ref{lem:universal} $\Phi$ also is a $t\tilde{p} +(1-t)\tilde{q}$--ladder. Now,  $d(p+q) \leq \tilde{p} +\tilde{q}  \leq d' (p+q)$ implies that properties (\ref{eq:pq})  also hold for $\tilde p$ and $\tilde{q}$.  Thus, for every $0<t < 1$, all conditions in 
Theorem~\ref{thm:erchler-kaimanovich} are satisfied. Therefore, $\mu_{t,\tilde p,\tilde{q}}$ and $\mu^{t,\tilde p,\tilde{q}}$ are non-Liouville.
\end{proof}

\section{Liouville measures for ICC Amenable Groups}\label{sec:liouville} 
We will show that, depending on how weights are assigned to F{\o}lner sets or switching sets, the probability measure $\mu_{t,p,q}$ (Definition~\ref{df:mu}) might become Liouville, particularly, this is the case when $\mu_{t,q,p}$ is not fully supported on all switching sets.  For this purpose, it is easier to  study $\mu_{t,p,q}$  in the following way. Let $(\xi_i,\zeta_i)_{i\geq 1}$ be a sequence of i.i.d. random variables  on $\NN\times \{0,1\}$ according to $\al_{t,p,q}$ (Definition~\ref{df:alpha}). 
We select  $h_i$ uniformly from $F_k$  when $(\xi_i,\zeta_i)=(k,0)$, and uniformly from $S_k$ when $(\xi_i,\zeta_i)=(k,1)$.  
It is clear that $(h_i)_{i\geq 1}$ is a sequence of  i.i.d. random variables  where $h_i$ is distributed according to $\mu_{t,p,q}$. This construction allows us to transfer properties of fit times (Definition~\ref{df:fittime}). So, we say $(h_i)_{i\geq 1}$ has a fit time at time $n$ for $m$, when the corresponding $(\xi_i,\zeta_i)_{i\geq 1}$ has a fit time at time $n$ for $m$.

When $\theta$ is a finite measure on a group $G$, we denote by $\theta^{*n}$ the $n$-th fold convolution of $\theta$ with itself. We denote by $\|\theta\|  = \sum_{g\in G} \theta(g)$ the $\ell^1(G)$ norm of $\theta$. 
\begin{theorem}[\bf infinitely many fit times $\to$ Liouville property]\label{thm:trivial}
Suppose that $\al_{t,p,q}$ (Definition \ref{df:alpha})  has infinitely many fit times. Then, the associated  $\mu=\mu_{t,p,q}$ (Definition~\ref{df:mu}) is Liouville. 
\end{theorem}
\begin{proof}
First, we will show that $\|g\mu^{*n}- \mu^{*n}\| \to 0$ as $n\to \infty$  for every $g\in G$.

Let $M_n=\max\{\xi_1,\dots,\xi_n\}$, where $(\xi_i,\zeta_i)_{i\geq 1}$ is a sequence of i.i.d. random variables based on 
$\al_{t,p,q}$ and $\varphi$ be non-constant and non-decreasing as in Lemma~\ref{lem:phi}. Let $T_n$ be the set of  increments of the random walk which do not have a fit time at time $n$ for $M_n$  or $M_n \leq \varphi_n$. % that is

By abuse of notation, we also denote by $T_n$ the image of $T_n$ under  $(h_i)_{i\geq 1} \to h_1\dots h_n$.
We split $\mu^{*n}$ according to $T_n$ and its complement $T^c_n$:
\begin{equation}\label{eq:split}
\mu^{*n} =\mu^{*n}|_{T_n^c} +  \mu^{*n} |_{T_n} .
\end{equation}
Note that $\mu^{*n}|_{T_n^c}$ can be written as
\begin{equation}\label{eq:combin}
\mu^{*n} |_{T_n^c}= \sum_{k_1,\cdots,k_n\in \NN} c_{k_1,\cdots,k_n} \beta_{k_1}*\cdots *\beta_{k_n}|_{T_n^c},
\end{equation}
where  each $\beta_{i_j}$ is either $\U(F_{i_j})$ or $\U(S_{i_j})$. 
However, the definition of $T^c_{n}$ implies that for every  $\beta_{k_1}*\cdots *\beta_{k_n}|_{T_n^c}$ there exist probability measures $\lambda$ and $\vartheta$ such that
$$
\beta_{k_1}*\cdots *\beta_{k_n}|_{T_n^c} = \lambda * \U(F_m) * \vartheta, 
$$
where $\varphi_n<m=\max\{k_1,\dots, k_n\}$ such that  $n < \Phi_m$ and
\begin{equation}\label{eq:etasup}
\mbox{ support of }  \lambda\subseteq \Big(\bigcup_{j=1}^{m-1} (S_j\cup F_j)\Big)^{n-1} \subseteq  \Big(\bigcup_{j=1}^{m-1} (S_j\cup F_j)\Big)^{5\Phi_m}.
\end{equation}
 
Thus by the properties of F{\o}lner--switching sets 
\begin{equation}\label{eq:eta1}
\| \lambda * \U(F_m)*\vartheta - \U(F_m)*\vartheta\| \leq  \| \lambda * \U(F_m) - \U(F_m)\| \leq \frac{1}m\cdot
\end{equation}
On the other hand, $\varphi_n<m $   for  every $(h_i)_{i\geq 1}$ in $T_n^c$. Thus,  
$$
g\in  \Big(\bigcup_{j=1}^{\varphi_n} (S_j\cup F_j)\Big)\subseteq  \Big(\bigcup_{j=1}^{m-1} (S_j\cup F_j)\Big).
$$
Also, the support of $g\lambda$ is at most $\Big(\bigcup_{j=1}^{m-1} (S_j\cup F_j)\Big)^n$ and $n<5 \Phi_m$,   so by properties of F{\o}lner--switching sets 
\begin{equation}\label{eq:eta2}
\| g \lambda*\U(F_m) - \U(F_m)\|  \leq \frac{1}{m} < \frac{1}{\varphi_n}\,.
\end{equation}
Combining (\ref{eq:eta1}) and (\ref{eq:eta2}) implies that
\begin{equation}\label{eq:eta3}
\lim_{n\to \infty} \|g\mu^{*n}|_{T_n^c}- \mu^{*n}|_{T_n^c}\| = 0 \ \ \ \forall g \in \Big(\bigcup_{j=1}^{\varphi_n} (S_j\cup F_j)\Big) \,.
\end{equation}
On the other hand, almost every increment  has infinitely many fit times and by Lemma~\ref{lem:phi}  $M_n > \varphi_n$ for all sufficiently large $n$. Thus, we can find a subsequence $(T_{n_j})_{j\geq 1}$ such that
  \begin{equation}\label{eq:lessep}
\lim_{j\to \infty} \PP(T_{n_j})= \|  \mu^{*{n_j}}|_{T_{n_j}} \| = 0 \,.
 \end{equation}
 We have 
\begin{equation}\label{eq:eta4}
 \|g\mu^{*n} - \mu^{*n}\| \leq   \|g\mu^{*n}|_{T_n^c} - \mu^{*n}|_{T_n^c}\| + 2 \|  \mu^{*n}|_{T_n} \| \,.
\end{equation}
Combining (\ref{eq:eta3}), (\ref{eq:lessep}), and (\ref{eq:eta4}) implies that a subsequence of  $\|g\mu^{*n} - \mu^{*n}\|$ converges to zero for every $g \in G$ as $\varphi_n$ is non-decreasing.  Because $\|g\mu^{*n} - \mu^{*n}\| $ is decreasing,  we  obtain  $\lim_{n\to\infty} \|g\mu^{*n} - \mu^{*n}\|=0$ for every $g\in G$. 
Now suppose that $f$ is bounded $\mu$--harmonic. Thus,
$$
|f(g)-f(e)| \leq \sum_{h\in G} |f(gh) \mu^{*n}(h) - f(h) \mu^{*n}(h)| \leq \|f\|_\infty \|g\mu^{*n}-\mu^{*n}\| \,.
$$
We have $\lim_{n\to\infty}\|g\mu^{*n}-\mu^{*n}\|=0$, thus $f$ must be constant, and $\mu$ is Liouville. 
\end{proof}
\begin{corollary}\label{cor:gaptri}
Suppose that $\al_{t,p,q}$ (Definition \ref{df:alpha}) has a long gap for $q$ (Definition~\ref{df:gap}). Then,  the associated probability measure $\mu_{t,p,q}$ (Definition \ref{df:mu}) is Liouville. 
\end{corollary}
\begin{proof}
Because $\al_{t,p,q}$ has a long gap for $q$, Theorem~\ref{them:longfit} implies that $\al_{t,p,q}$ has infinitely many fit times. Thus, Theorem~\ref{thm:trivial} implies $\mu_{t,p,q}$ is Liouville. 
\end{proof}

\subsection{Convex combinations of Liouville measures}
We  are now ready to prove  Theorem~\ref{thm:A}.

\begin{theorem}[=Theorem~\ref{thm:A}]\label{thm:main1}
Let $G$ be a countable non-hyper-FC-central amenable group. For every $k\geq 2$,  there exists a sequence of symmetric and non-degenerate probability measures $(\mu_m)_{m\geq1}$ such that $\mu = \sum_{m\geq 1} b_m \mu_m$ is non-Liouville if and only if the probability measure $b$ on $\NN$ has at least $k$ atoms. 
\end{theorem}
\begin{proof} Note that a non-degenerate probability measure  on $G$ is Liouville if and only if its image in $G$ mod hyper-FC of $G$ is Liouville \cite[Proposition~5.10]{Erschler-Kaimanovich}.
 Thus without  loss of generality, we can assume $G$ is a countable amenable ICC group.  Let $p$ and $q$ be fully supported probability measures on $\NN$ that satisfy condition (\ref{eq:pq}). Suppose that $(F_n,S_n,\Phi_n)_{n\geq 1}$ is F{\o}lner--switching where $\Phi$ is a  universal $q$--ladder as in Lemma~\ref{lem:universal}. Let $(A_j)_{j\geq 1}$ be a $k$--cover as in Lemma~\ref{lem:conditionalgap}.  Let 
$p^j=p^{A_j}$, $t_j =p(A_j)$, and   
\begin{equation}\label{eq:mui}
    \mu_j= \mu_{t_j,p^j,q} = t_j \sum_{n=1}^\infty p^j_n \U(S_n) + (1-t_j) \sum_{n=1}^\infty q_n \U(F_n) \,.
\end{equation}
Since $\Phi$ is a universal $q$--ladder and $\al_j=t_j p^j+(1-t_j)q \geq (1-t_j)q$, by Lemma~\ref{lem:universal} $\Phi$ is also an $\al_j$--ladder. By Lemma~\ref{lem:conditionalgap}, every $k-1$ convex combination of $\al_j$ has a long gap for $q$, thus Corollary~\ref{cor:gaptri} implies that every $k-1$ convex combination of $\mu_j$ is Liouville. 

Conversely, let $b$ be a probability measure on $\NN$ with at least $k$ atoms and $\mu=\sum_{i=1}^\infty b_i \mu_i$, where $\mu_i$ is defined as in (\ref{eq:mui}). It is easy to see that $\mu =  \mu_{c,\tilde{p},q} $,  where 
$$
c= \sum_{i=1}^\infty b_i t_i,\ \ \ c_n = \sum_{i=1}^\infty b_i 1_{A_i}(n), \ \   \ \ \ \tilde{p}_n = \frac{c_n p_n}{c} \,.
$$ 
Recall that $(A_i)_{i\geq 1}$ is a $k$--cover, so the union of $k$ distinct sets $A_i$s is $\NN$.  Because $b$ has at least $k$ atoms,  the probability measure $\tilde{p}$ is fully supported on $\NN$.  Let 
$ \bar b = \bar b_1+ \dots +\bar b_{k-1}$ where $ \bar b_1\geq \dots \geq \bar b_{k-1}$ are the first $k-1$ largest values of $\{b_i: i\in \NN\}$. We claim $c_n \geq 1- \bar b$ for every $n\geq 1$. Suppose that $n\not\in (A_{l_1} \cup \dots \cup A_{l_{k-1}})$ for some $l_1,\dots,l_{k-1}$. Because $(A_i)_{i\geq1}$ is a $k$--cover, we conclude $n\in A_i$ for every $i \not \in \{l_1,\dots,l_{k-1} \}$. Thus $c_n \geq 1-b_{l_1}- \cdots - b_{l_{k-1}} \geq 1- \bar{b}$. Therefore, $(1-\bar b)p_n \leq c \bar p_n \leq p_n$, and by Lemma~\ref{lem:nont} $\mu=\mu_{c,\tilde{p},q}$ is non-Liouville.
\end{proof}

\section{Liouville and non-Liouville with Finite Entropy}\label{sec:entropy}
We recall basic results regarding entropy. Let $X$ be a countable random variable on a probability space $(\Omega, P)$. Then the entropy of $X$
is defined as
$$
H(X) = -\sum_{x \in \Omega} P(X=x)  \log P(X=x) \,,
$$
where $0\log 0$ is defined as 0. 
Suppose that $Y$ is a countable random variable. Then, the conditional entropy of a given outcome $y$ is defined as 
$$
 H(X|Y=y) = -\sum_{x\in X} P(X=x|Y=y) \log P(X=x|Y=y)\,,
$$ and the conditional entropy of $X$ given $Y$ is defined as
$$
H(X | Y) = \sum_{y} H(X |Y=y) P(Y=y)\,.
$$

\begin{lemma}\label{lem:basic-entropy}
    Let $X$ and $Y$ be two discrete random variables with finite entropy, then 
 $$H(X|Y) \leq H(X) \leq H(X|Y) + H(Y) \,.
 $$
\end{lemma}
Let $\mu$ be a probability measure on a countable group $G$. Let $w_n=g_1\cdots g_n$ where $g_1,\dots,g_n$ are i.i.d. according to $\mu$. Thus the distribution of $w_n$ is $\mu^{*n}$, and we write $H(\mu^{*n})=H(w_n)$. It is straightforward to check the subadditivity of entropy, $H(\mu^{*(m+n)}) \leq H(\mu^{*n}) + H(\mu^{*m})$. Thus, the asymptotic entropy of $\mu$ is defined as 
$$
h(\mu) = \lim_{n\to\infty} \frac{1}{n}H(\mu^{*n}) = \inf_n \frac1n H(\mu^{*n}).
$$
The following classic entropy result will be used to determine whether a probability measure with finite entropy is Liouville.
\begin{theorem}\cite[Theorem 1.1]{K-Vershik83}\cite[Theorem, p.~255]{Der80}
Let $\mu$ be a probability measure with finite entropy on a countable group $G$. Then, $\mu$ is Liouville if and only if $h(\mu)=0$.
\end{theorem}
We denote by $\P(A)$ (and $\P_{sym}(A)$) the sets of probability measures (and symmetric probability measures) on $A$.
\begin{definition}[\bf $\D$--metric] \label{df:dmetric}
Let $\lambda$ and $\theta$ be two probability measures on a countable group $G$. We define the $\D$-metric between  $\theta$ and $\lambda$ by
$$
\D(\theta, \lambda) = \| \theta - \lambda\| +  |H(\theta) -H(\lambda)| \,.
$$
It is straightforward to check that $\D$ is a metric on $\P(G)$. 
\end{definition}

\begin{lemma}\label{lem:continuity}
Let $G$ be a countable group and $n\in \NN$. The map $
\lambda \to H(\lambda^{*n})$ over $\P(G)$ is continuous with respect to $\D$. 
\end{lemma}
\begin{proof}
Suppose that $\lambda \in \P(G)$. If $n=1$, the proof is trivial. Suppose that $n>1$. We use the existence of maximal (optimal) coupling. Suppose that $\theta \in \P(G)$ such that $0<2t=\|\theta - \lambda\|<2$. 
Let 
$$
\tilde{w}=\min\{\lambda,\theta\}\,, \hspace{1cm} \tilde{\lambda} = \lambda-\tilde{w}\,,\hspace{1cm}\tilde{\theta}=\theta-\tilde{w} \,.
$$

We now define probability measures $w=\dfrac{\tilde{w}}{\|\tilde{w}\|}, \lambda_0= \dfrac{\tilde{\lambda}}{\|\tilde{\lambda}\|} $, and $\theta_0= \dfrac{\tilde{\theta}}{\|\tilde{\theta}\|}$. Thus,
$$
\lambda = t \lambda_0 + (1-t) w, \hspace{1cm} \theta = t \theta_0 +(1-t) w \,.
$$
 Let $ b_1^t $ be a Bernoulli random variable such that $ b_1^t = 1 $ corresponds to sampling according to $ \theta_0 $ with probability $ t $, and $ b_1^t = 0 $ corresponds to sampling according to $ w $ with probability $ 1 - t $.

We now estimate $H(\theta^{*n})$. Let $b^t = b_1^t+ \dots +b_n^t$, where $b_1^t,\dots,b_n^t$ are i.i.d. according to $b_1^t$. So, $b^t$ has the Binomial distribution counting the number of times sampling according to $\theta_0$.
Using the properties of entropy in Lemma~\ref{lem:basic-entropy}, 
$$
H(\theta^{*n}) \leq H(\theta^{*n}|b^t \not = 0) P(b^t \not =0) + H(\theta^{*n}|b^t = 0) P(b^t =0) + H(b^t) \,.
$$
By subadditivity, 
$$
H(\theta^{*n}|b^t \not = 0) \leq n H(\theta) \,, \ \ H(b^t) \leq n H(b^t_1) \,.
$$
Since $H(\theta^{*n}|b^t=0) P(b^t=0)=(1-t)^nH(w^{*n})\leq \min\{H(\lambda^{*n}), H(\theta^{*n})\}$, we conclude that 
\begin{equation}\label{eq:thetan}
H(\theta^{*n}) \leq  n(1-(1-t)^n)H(\theta) + H(\lambda^{*n})+n H(b_1^t).
\end{equation}
Note that we can write a similar upper-bound for $H(\lambda^{*n})$:
\begin{equation}\label{eq:2}
H(\lambda^{*n}) \leq  n(1-(1-t)^n)H(\lambda) + H(\theta^{*n})+n H(b_1^t) \,. 
\end{equation}
After using $H(\theta)  \leq \D(\theta, \lambda) +H(\lambda)$ and $1-(1-t)^n\leq 1$ in (\ref{eq:thetan}) and considering (\ref{eq:2}),
we obtain
$$
|H(\theta^{*n}) - H(\lambda^{*n}) |\leq n \D(\theta, \lambda)   + n(1-(1-t)^n )H(\lambda) +n H(b_1^t) \,.
$$
Note that $H(b^t_1)= -t \log t -(1-t) \log (1-t)$ and $n(1-t)^n H(\lambda)$ go to zero as $t \to 0$, which imply the map $\lambda \to H(\lambda^{*n})$ is continuous with respect to the $\D$--metric.
\end{proof}
\begin{definition}
    A group $G$ is called \emph{finitely Liouville} when any finitely supported symmetric probability measure on $G$ is Liouville. 
\end{definition}
For example, groups with subexponential growth, the lamplighter groups over $\ZZ$ and $\ZZ^2$, and  $S_\infty$, the infinite symmetric group of finite permutations of $\NN$,   are finitely Liouville.
\begin{lemma}\label{lem:uniform}
    Let $G$ be finitely Liouville  and $V$ a symmetric finite subset of $G$. Then, for every $\ep>0$ there exists $N$ (depending on $V$ and $\ep$)  such that 
    $$
    H(\mu^{*N}) < \ep N
    $$ 
    for every symmetric probability measure $\mu$ on $V$.
\end{lemma}
\begin{proof}
    Let $\ep>0$ and $n \in \NN$.  Consider the set
    $$
    B(n\ep) = \{ \mu \in \P_{sym}(V)\ :\ H(\mu^{*n})<n\ep \},
    $$ 
    which 
    is an open set with respect to $\D$--metric by Lemma~\ref{lem:continuity}. Since $G$ is finitely Liouville, $(B(n\ep))_{n\geq 1}$ is an open cover for $\P_{sym}(V)$. Since $V$ is finite, $\P_{sym}(V)$ is compact with respect to $\D$. Thus, there exists $N$ such that 
    $$ 
    \P_{sym}(V) =\bigcup_{n=1}^N B(n\ep) \,.
    $$
    We claim that $B(i\ep) \subseteq B(j\ep)$ for every $i\leq j$, and consequently, $P_{sym}(V) = B(N\ep)$.  For the sake of contradiction, suppose that $H(\mu^{*n})< n \ep$ but $H(\mu^{*(n+1)}) \geq (n+1)\ep$. By subadditivity, $H(\mu^{*(n+1)}) \leq (n+1)H(\mu)$, we obtain $H(\mu) \geq \ep$. On the other hand, $(H(\mu^{*n+1})-H(\mu^{*n}))_{n\geq 1}$ is a decreasing sequence, see \cite[Proposition 1.3]{K-Vershik83}, thus 
    $$
    H(\mu^{*2}) - H(\mu) \geq H(\mu^{*(n+1)}) - H(\mu^{*n}) \geq \ep \,.
    $$
    Thus, $H(\mu^{*2}) \geq 2\ep$. Similarly, we can show that $H(\mu^{*i}) \geq i \ep$ for $i=1,\dots, n+1$, which is a contradiction with $H(\mu^{*n}) <n \ep$. Therefore, $B(n\ep) \subseteq B((n+1)\ep)$ for every $n\in\NN$ and $\ep>0$. Thus, $\P_{sym}(V)= B(N\ep)$.
\end{proof}
\begin{remark}\label{re:cont}
    Note that when $V'$, a set of probability measures,  admits a \emph{uniform first moment}, then  the map $V' \ni\mu\to H(\mu^{*n})$ is continuous in $\ell^1(G)$ \cite[Lemma A.1.]{Tanaka}. Under additional conditions on the group, the continuity of $\mu \to h(\mu) $ has been established in \cite{Kaimanovich-Erschler2013} and \cite{silva2025continuity}. However, the advantage of working with $\D$--metric is that no moment condition is needed for continuity $\mu \to H(\mu^{*n})$. 
\end{remark}
Denote by $\D(\theta,P_{sym}(V))=\inf\{\D(\theta,\lambda) \ :\ \lambda \in \P_{sym}(V) \}$, where  $\theta\in \P(G)$ and  $V$ is a symmetric subset $V$ of $G$. 
\begin{lemma}\label{lem:neighbor} Let $G$ be finitely Liouville. Then for every $\ep>0$ and every finite  symmetric subset $V$ of $G$ there exists  $\delta(V,\ep)>0$ such that 
$$
h(\theta)<\ep
$$ 
for every $\theta \in \P(G)$
satisfying $\D(\theta, \P_{sym}(V)) < \delta(V,\ep)$. 
\end{lemma}
\begin{proof}
    Suppose that $N$ is as in Lemma~\ref{lem:uniform} for  $\ep/2>0$ and the finite symmetric set $V$. Let  $\lambda \in \P_{sym}(V)$, thus by Lemma~\ref{lem:uniform} 
    \begin{equation}\label{eq:uniN}
    H(\lambda^{*N})<  \frac{N\ep}2 \,.
    \end{equation}
    On the other hand, by Lemma~\ref{lem:continuity}  for $N\ep/2$ there exists $\delta_\lambda>0$ such that 
    \begin{equation}\label{eq:lemuni}
    \forall \theta \in\P(G) \mbox{ \hspace{.3cm}  if  \hspace{.3cm} }\D(\theta,\lambda)< \delta_\lambda \mbox{\hspace{.5cm} then \hspace{.3cm}}
    H(\theta^{*N}) < H(\lambda^{*N}) + \frac{N\ep}2  \,.
    \end{equation}
   For $\lambda \in \P_{sym}(V)$ consider the open set
    $$
    B(\lambda)=\Big\{ \theta \in \P(G) : \D(\theta,\lambda)< \frac{\delta_\lambda}2 \Big\}\,.
    $$
    Because $\P_{sym}(V)$ is compact, for some $j\geq 1$
    $$
    \P_{sym}(V) \subseteq \bigcup_{i=1}^j B(\lambda_i)\,.
    $$
 Define $\delta(V,\ep) = \min\{\delta_i/2 : i=1,\dots,j\}.$  Suppose that $\theta\in \P(G)$ such that $\D(\theta,\lambda) < \delta(V,\ep)$ for some $\lambda \in P_{sym}(V)$, then $\lambda \in B(\lambda_i)$ for some $1 \leq i \leq j$ and by a triangle inequality $\D(\theta,\lambda_i) <\delta_{\lambda_i}$. Therefore, applying (\ref{eq:uniN}), (\ref{eq:lemuni}) for $\lambda_i$, and $h(\theta) = \inf_n \frac1nH(\theta^{*n})$ yields 
 $$
h(\theta) \leq \frac{ H(\theta^{*N})}{N} <   \frac{H(\lambda^{*N})}{N} +\frac{\ep}2<\frac{\ep}{2} + \frac{\ep}2 = \ep\,.
 $$
Thus, we obtained the desired result. 
\end{proof}
For a sub-probability measure $\theta$, define 
$H(\theta) = - \sum_{g\in G} \theta(g) \log \theta(g)$. 

\begin{lemma}\label{lem:mui construction}
   Suppose that $G$ is finitely Liouville. Let $\mu$ be a fully supported, symmetric probability measure with finite entropy on $G$. Let $C_0>0$ be a constant. Then, there exist $(V_n)_{n\geq 1}, $ a disjoint sequence of finite symmetric subsets of $G$ such that $G= \bigcup_{n\geq 1} V_n$ and decreasing sequence $(r_n)_{n\geq 1}$ such that $0<r_n\leq 1$ so that for every $n\geq 1$ the sub-probability measure $\tm_n = \mu|_{V_n}$, the restriction of $\mu$ to $V_n$, satisfies  the following properties: 
   \begin{enumerate}
       \item $\|\tm_1\| \geq \frac12$,
       \item $C_0\displaystyle\sum_{i= n+2}^\infty \|\tm_i\| < \delta_n$,
       \item $C_0\displaystyle\sum_{i=n+2}^\infty H(\tm_i)  < \delta_n$,
       \item $C_0r_{n+1} \|\tm_{n+1}\|  < \delta_n$,
       \item $C_0 r_{n+1} H(\tm_{n+1})  < \delta_n$,
   \end{enumerate}
   where $\delta_n = \delta(V=V_1\cup\dots\cup V_n, \ep=\dfrac1n)$ is decreasing as in Lemma~\ref{lem:neighbor}.
\end{lemma}
\begin{proof}
    Let $\mu= \sum_{i=1}^\infty \tti_i$ where each $\tti_i$ is a finitely supported and symmetric sub-probability measure. It is clear that we can choose $\tti_i$'s with disjoint supports. We choose $m_1\geq 1$ large enough that $\sum_{i=1}^{m_1}  \|\tti_i\| \geq \dfrac12.$   Let $\tm_1=\sum_{i=1}^{m_1}\tti_i$ and $V_1$ be the support of $\tm_1$.  Suppose that by induction, we constructed  $V_1,\dots, V_n$, $m_1<m_2<\dots m_n$, and $\tm_1,\dots, \tm_n$ such that 
    $C_0\sum_{i>m_{(j+1)}}\|\tti_i\| <\delta_j$ and $C_0 \sum_{i>m_{(j+1)}} H(\tti_i) < \delta_j$ hold, and 
    $$
   \tm_{j+1} = \sum_{i>m_j}^{m_{j+1}} \tti_i  \hspace{1cm} j=1,\dots,n-1 \,.
    $$
     We now construct $\tm_{n+1}$. Let  $\delta_n=\delta(V_1\cup \dots \cup V_n, 1/n)$ be as in  Lemma~\ref{lem:neighbor} for $V=V_1\cup\dots\cup V_n$ and $\ep = 1/n$ such that $\delta_n<\delta_{n-1}$. Now, we choose   $m_{n+1}$ large enough such that $C_0\sum_{i>m_{(n+1)}} \| \tti_i\| <\delta_n$ and $C_0 \sum_{i>m_{(n+1)}} H(\tti_i) <\delta_n$. Let   
    $$
    \tm_{n+1} = \sum_{i>m_n}^{m_{n+1}} \tti_i \,, \hspace{1cm} V_{n+1} = \mbox{ support of } \tm_{n+1} \,.
    $$
    It is straightforward to see $(\mu_n)_{n\geq 1}$ satisfies properties (1)-(3). 
    Finally, we can choose $(r_n)_{n\geq 1}$ small enough that properties (4) and (5) are satisfied.
\end{proof}

\begin{lemma}\label{lem:zeronu}
    Let $G$ be finitely Liouville and $\mu$ a symmetric probability measure with finite entropy on $G$. Suppose that sequences $(r_i)_{i\geq1}$ and $(\tm_i)_{i\geq1}$ are constructed as in Lemma~\ref{lem:mui construction} for $C_0=30H(\mu) +30$. Let  $\nu=\sum_{i\geq 1} c_i \tm_i$ be a probability measure on $G$ such that $c_1\geq 1$ and $0 \leq c_i \leq 2$ for every $i\geq 1$. If $c_{n+1} \leq 2r_{n+1}$ for some $n\geq 1$, then 
    $$
    h(\nu) < \frac1n\,.
    $$ 
\end{lemma}
\begin{proof}
Let  $\tilde{\lambda} = \sum_{i=1}^{n} c_i \tm_i$ and  $\lambda=\tilde\lambda/\|\lambda\|$. Let $V$ be the support of $\tm_1+\dots+\tm_n$. It is enough to show that $\D(\nu,\lambda)<\delta_n=\delta(V,1/n)$, then the results follows immediately by Lemma~\ref{lem:neighbor}. 
    Note that 
    $$
    H(\nu) = \sum_{i=1}^\infty H(c_i\tm_i) = \sum_{i=1}^\infty c_i H(\tm_i)  - \sum_{i=1}^\infty \|\tm_i\| c_i \log c_i \,.
    $$
    Since  $H(\mu)$ is finite and for $0 \leq c_i \leq 2$ for $i\geq 1$, 
    \begin{equation}\label{eq:finitenu}
    H(\nu) \leq 2 \sum_{i=1}^\infty H(\tm_i) + 2\log 2 < 2H(\mu) + 1\,.
    \end{equation}
    On the other hand, 
   $$
 H(\lambda) = \frac{1}{\|\tilde\lambda\|} \sum_{i=1}^{n} H(c_i \tm_i) + \log\|\tilde\lambda\|\,.
    $$
    Note that $\frac12 \leq \|\tilde \lambda\| \leq 1$, thus
\setlength{\jot}{5pt} 
\begin{align}
|H(\nu) - H(\lambda)| \leq  & \Big( \frac{1}{\|\tilde\lambda\|}-1\Big)  \sum_{i=1}^{n} H(c_i \tm_i) -\log\|\tilde\lambda\| +\sum_{i=n+1}^{\infty} H(c_i \tm_i)  \\ 
 \label{eq:differenceentropy} \leq &  2 H(\nu) \Big(1-\|\tilde\lambda\|\Big) + H(c_{n+1}\tm_{n+1})+\sum_{i=n+2}^\infty H(c_{i}\tm_{i}) \,.
\end{align}
The inequality in (\ref{eq:differenceentropy}) stems from $x^{-1}-1-\log x < 2(1-x)$ when $1/2 \leq x \leq 1$. This is where we used $c_1\geq 1$ and $\|\tm_1\|\geq 1/2$ to ensure $\|\tilde\lambda\|\geq 1/2$. 
Using $c_{n+1} \leq 2r_{n+1}$ and properties of $(\tm_i)_{i\geq 1}$ and $r_{n+1}$ as in Lemma~\ref{lem:mui construction}, we can bound $|H(\nu)-H(\lambda)|$ further:
\setlength{\jot}{10pt} 
\begin{align}\label{eq:differenceentropy2}
1- \|\tilde\lambda\| \leq & \sum_{i=n+1}^\infty c_i\|\tm_i\| < \frac{4\delta_n}{C_0}   \\
H(c_{n+1}\tm_{n+1}) \leq & 2r_{n+1}H(\tm_{n+1}) + 2 r_{n+1} \|\tm_{n+1}\| < \frac{4\delta_n}{C_0} \\
\sum_{i=n+2}^\infty H(c_i\tm_i) \leq &\sum_{i=n+2}^\infty 2H(\tm_i) +\sum_{i=n+2}^\infty \|\tm_i\| \ c_i  |\log c_i| < \frac{3\delta_n}{C_0} 
\end{align}
Putting the above inequalities implies 
$$
|H(\nu)-H(\lambda)| <\frac{8}{5C_0}(2H(\mu)+1)\delta_n + \frac{7}{C_0}\delta_n\,.
$$
On the other hand, 
$$
\|\nu-\lambda\| = 2(1-\|\lambda\|) \leq \frac{8}{C_0}\delta_n\,.
$$ 
Note that $C_0= 30H(\mu)+30$. Therefore, $\D(\nu,\lambda)=\|\nu-\lambda\|+|H(\nu)-H(\lambda)|<\delta_n$.
\end{proof}
We  are now prepared to prove Theorem~\ref{thm:B}. 
 \begin{theorem}\label{thm:entropy}
Let $G$ be  finitely Liouville  non-hyper-FC-central group, and $k\geq 2$. Then there exists a sequence of non-degenerate symmetric probability measures $(\nu_n)_{n\geq 1}$ 
 on $G$ with finite entropy  such that  any $k-1$ convex combination of them is Liouville and any non-trivial $k$ or more (finite or infinite) convex combination of them is non-Liouville. 
 \end{theorem}
\begin{proof}
Similar to Theorem~\ref{thm:main1}, without loss of generality, we assume $G$ is a countable amenable ICC group. Let $p$ and $q$ be two fully supported probability measures on $\NN$ with finite entropy. Suppose that $p$ and $q$ satisfy property (\ref{eq:pq}) in Lemma~\ref{lem:nont}. Let $0<t<1$ and $\mu=\mu^{t,p,q}$ as in Definition~\ref{df:mu2}. Thus, $\mu$ is a fully supported, symmetric probability measure on $G$  with finite entropy. Suppose that  $(r_i)_{i\geq 1}$ and $(\tm_i)_{i\geq1}$ are as in Lemma~\ref{lem:mui construction} for probability measure $\mu$. 
Let $(A_n)_{n\geq1}$ be a $k$--cover (see Lemma~\ref{lem:kcovering}), without loss of generality (after including 1 to every $A_n$) we can assume $1\in A_n$ for $n\geq 1$. Define 
$$
\tn_n = \sum_{i \in A_n} \tm_i + \sum_{i \not\in A_n} r_i \tm_i\,, \hspace{1cm} \nu_n = \frac{\tn_n}{\|\tn_n\|} \,.
$$
Note that $\nu_n$ is fully supported and symmetric. Suppose that $b$ is a probability measure on $\NN$. Thus  every $\tm_i$ appears with some positive coefficients in 
\begin{equation}\label{eq:cimu}
 \nu = \sum_{n=1}^\infty b_n \nu_n = \sum_{i=1}^\infty c_i \tm_i \,.   
\end{equation}
Note that $1 \in A_n$ and $\|\tm_1\|\geq \dfrac12$, thus $\|\tn_n\| \geq \|\tm_1\|\geq \dfrac12$. On the other hand $r_i\leq 1$,  thus 
\begin{equation}\label{eq:ttn_bound}
   \frac12 \leq \|\tn_n\| \leq 2\,. 
\end{equation}
\subsection*{$k-1$ convex combinations is Liouville:}
Without loss of generality, assume $b_n=0$ for $n\geq k$. It is clear that $0< c_i \leq 2$ for every $i\geq 1$. If $i\not\in \bigcup_{n=1}^{k-1} A_n$, then by (\ref{eq:ttn_bound})
$$
c_i= \sum_{n=1}^{k-1}  \frac{b_nr_i}{\|\tn_n\|} \leq 2r_i \,.
$$ 
Therefore, by Lemma~\ref{lem:zeronu} $h(\nu) <\dfrac{1}{i-1}$.  Since $(A_n)_{n\geq1}$ is a $k$--cover, $\NN\backslash (\bigcup_{n=1}^{k-1} A_n)$  is an infinite set. Thus, $h(\nu)=0$, and $\nu$ is Liouville.
\subsection*{At least $k$ convex combination is non-Liouville:}
Now suppose that at least $k$ of $\{b_n\ :\ n\geq 1\}$ are non-zero. Without loss of generality, suppose that $b_1b_2\cdots b_k>0$. Let $\tilde b=\min\{b_1,b_2,\dots,b_k\}$. Since $\NN=\bigcup_{n=1}^k A_n$, for every $i\in \NN$ there exists some $n \in \{1,\dots,k\}$ such that 
$$
c_i \geq \frac{b_n}{\|\tn_n\|} \geq \frac{\tilde b}{2}>0 \,.
$$
On the other hand, by construction of $(\tm_i)_{i\geq1}$ in Lemma~\ref{lem:mui construction}, 
$$
\mu^{t,p,q}= \sum_{i=1}^\infty \tm_i \,.
$$
Thus, $\nu = \mu^{\bar{t},\bar{p},\bar{q}}$ for some $0<\bar{t}<1$, and some fully supported probability measures $\bar{p}$ and $\bar{q}$ on $\NN$. Since $\dfrac{\tilde{b}}{2} \leq c_i \leq 2$, we obtain $\dfrac{\tilde{b}}{2} \bar{p} \leq p \leq 2 \bar{p}$  and $\dfrac{\tilde{b}}{2} \bar{q} \leq q \leq 2 \bar{q}.$ Thus, Lemma~\ref{lem:nont} implies that $\nu$ is non-Liouville. 
\end{proof}
 \bibliographystyle{amsalpha}
\bibliography{ref}
\end{document}